\documentclass[reqno,12pt]{article}
\usepackage{amsmath,amsthm,amssymb,bm} 
\usepackage{datetime}

\renewcommand{\phi}{\varphi}
\renewcommand{\epsilon}{\varepsilon}

\newtheorem{theorem}{Theorem}
\newtheorem{lemma}[theorem]{Lemma}
\newtheorem{proposition}[theorem]{Proposition}
\newtheorem{corollary}[theorem]{Corollary}

\newcommand\E{{\mathbb E}}

\renewcommand\P{{\mathbb P}}

\newcommand\Bi{\operatorname{Bi}}

\begin{document}
\title{Packing random graphs and hypergraphs}
\author{
B\'ela Bollob\'as%
\footnote{Department of Pure Mathematics and Mathematical Statistics, Wilberforce Road, Cambridge CB3
0WB, UK; {\em and}  Department of
Mathematical Sciences, University of Memphis, Memphis TN38152,
USA; {\em and} 
London Institute for Mathematical Sciences, 35a South St, Mayfair,
London W1K 2XF, UK; email: bb12@cam.ac.uk.  Research supported in part
by NSF grant ITR 0225610; and by MULTIPLEX no. 317532.} 
\and
Svante Janson%
\footnote{Department of Mathematics, Uppsala University, PO Box 480,
SE-751~06 Uppsala, Sweden;
email: svante.janson@math.uu.se.
Research supported in part by the Knut and Alice Wallenberg Foundation.
}
\and
Alex Scott%
\footnote{
Mathematical Institute,
University of Oxford,
Andrew Wiles Building,
Radcliffe Observatory Quarter,
Woodstock Road,
Oxford OX2 6GG, UK;
email: scott@maths.ox.ac.uk.}
}
\date{23 August, 2014; revised 24 February 2016}
\maketitle

\begin{abstract}
We determine to within a constant factor the threshold for the property that two random $k$-uniform hypergraphs
with edge probability $p$ have an edge-disjoint packing into the same vertex set.  More generally, we allow the hypergraphs to have different densities.
In the graph case, we prove a stronger result, on packing a random graph with a fixed graph.
\end{abstract}

\section{Introduction}

Let $G_1$ and $G_2$ be two $k$-uniform hypergraphs 
of order $n$.  We say that $G_1$ and $G_2$ can be \emph{packed} if they can
be placed onto the same vertex set so that their edge sets are disjoint.   

In the graph case, quite a lot is known.  
Bollob\'as and Eldridge \cite{BE76} and Catlin \cite{C74} independently conjectured that if 
$(\Delta(G_1)+1)(\Delta(G_2)+1)
\le n+1$ then $G_1$ and $G_2$ can be packed.
Sauer and Spencer \cite{SS78}
proved that graphs $G_1$ and $G_2$ 
of order $n$ can be packed if $\Delta(G_1)\Delta(G_2)<n/2$.
Let us note that the conjectured bound would be tight: suppose that $n=ab-2$, and let 
$G_1=(b-1)K_a \cup K_{a-2}$ (the vertex-disjoint
union of $b-1$ complete graphs of order $a$ and a complete graph of order $a-2$)
and $G_2=(a-1)K_b\cup K_{b-2}$.
Then $(\Delta(G_1)+1)(\Delta(G_2)+1)=n+2$,
but $G_1$ and $G_2$ cannot be packed.

For fixed $k\ge3$,
the graph example given above is easy to generalize: 
suppose that $n=(a-1)(b-1)(k-1)+a+b-3$.  Let 
$G_1$ be the vertex-disjoint
union of $b-1$ complete $k$-uniform graphs of order $(a-1)(k-1)+1$ and $a-2$ isolated vertices;
let 
$G_2$ be the vertex-disjoint
union of $a-1$ complete $k$-uniform graphs of order $(b-1)(k-1)+1$ and $b-2$ isolated vertices.
Then $\Delta(G_1)\Delta(G_2)=\Theta(a^{k-1}b^{k-1})=\Theta(n^{k-1})$,
but $G_1$ and $G_2$ cannot be packed.  
For another family of examples, choose $r<k$ and fix an $r$-set
$A\subset[n]$.  Let $G_1$ have all edges containing $A$,
and $G_2$ be an $(n,k,r)$-design (these are now known to exist for suitable $n$: see Keevash \cite{K14}).
$G_1$ and $G_2$ cannot be packed, and we have $\Delta(G_1)=\Theta(n^{k-r})$ and $\Delta(G_2)=\Theta(n^{r-1})$, and so again
$\Delta(G_1)\Delta(G_2)=\Theta(n^{k-1})$.
On the positive side, much less is known.
Teirlinck \cite{T77} (see
Alon \cite{A94} for further results and discussion)
showed that, for $n\ge7$, any two Steiner triple systems $G_1$, $G_2$ can be
packed: note that these satisfy
$\Delta(G_1)\Delta(G_2)=\Theta(n^2)$.  
There are also some nice results when one of $G_1$ and $G_2$ has very small maximal degree: see 
R\"odl, Ruci\'nski and Taraz \cite{RRT} and 
Conlon \cite{C}.

In this paper, we consider what happens when $G_1$ and $G_2$ are {\em random} hypergraphs.
For integers $k,n$ and $p\in[0,1]$, we write $\mathcal G(n,k,p)$ for the random $k$-uniform hypergraph 
on $n$ vertices in which each possible edge is present indepedently with probability $p$; 
when $k=2$, we write $\mathcal G(n,p)=\mathcal G(n,2,p)$.
In the graph case, with $G_1\in\mathcal G(n,p)$ and $G_2\in\mathcal G(n,q)$, the extremal results
mentioned above suggest that we should expect a condition of form $pqn\le c$  (for suitable $c$) to be able to pack $G_1$ and $G_2$.  
More generally, for $k$-uniform hypergraphs, we might hope for a condition of form $pqn^{k-1}\le c$, as this would give 
$\Delta(G_1)\Delta(G_2)=O(n^{k-1})$ with high probability
(i.e.~with probability $1-o(1)$ as $n\to\infty$)
provided $p$, $q$ are not extremely small (for instance $\min\{p,q\}\gg\log n/n^{k-1}$ is enough).
In fact, we shall show here that it is possible to pack rather denser graphs: if $G_1$ and $G_2$ are both random then we can allow an additional factor 
$\log n$ in the product $pqn^{k-1}$, but not more.
(We note that a similar phenomenon occurs when we try to minimize the overlap of two random hypergraphs: see 
Bollob\'as and Scott \cite{BS} and
Ma, Naves and Sudakov \cite{MNS}.)

We will prove the following theorem.

\begin{theorem}\label{upfront}
Let $\delta\in(0,1)$.  For every $k\ge2$, there exists $\epsilon>0$ such that the following holds.  Let $p=p(n)$ and $q=q(n)$ be positive reals such that
  \begin{itemize}
   \item $\max\{p,q\}\le 1-\delta$
   \item $pq\le\epsilon\log n /n^{k-1}$.
  \end{itemize}
Let $G_1\in\mathcal G(n,k,p)$ and
$G_2\in\mathcal G(n,k,q)$ be random $k$-uniform hypergraphs of order $n$.  Then, with high probability, 
there is a packing of $G_1$ and $G_2$.
\end{theorem}

Note that if $pq=\epsilon\log n /n^{k-1}$
then with high probability $G_1$ and $G_2$ satisfy 
$\Delta(G_1)\Delta(G_2)=\Theta(n^{k-1}\log n)$.

The bound on $pq$ in Theorem \ref{upfront} is easily seen to be sharp up to the constant.  Indeed, if
$G_1\in\mathcal G(n,k,p)$ and
$G_2\in\mathcal G(n,k,q)$ then the probability that $G_1$ and $G_2$ can be packed is at most
the expected number of packings
$$n!(1-pq)^{\binom nk}\le \exp(n\log n - (1+o(1))pqn^k/k!)
$$
which is $o(1)$ if $pq\ge\alpha\log n/n^{k-1}$ for any constant $\alpha>k!$.  In particular, if we take $p=q$, then 
combining this bound with Theorem \ref{upfront} shows
that the threshold density for two random $k$-uniform hypergraphs to be unpackable is $\Theta(\sqrt{\log n/n^{k-1}})$.

In the case of graphs, we will in fact prove a much stronger result: it turns out that we can take just {\em one} of the two graphs to be random.
Indeed, we prove the following.

\begin{theorem}\label{extend}
For all $\gamma, K>0$ and $\delta\in(0,1)$ there exists $\epsilon>0$ such that the following holds.  Let $p=p(n)$ and $q=q(n)$
be positive reals such that
  \begin{itemize}
   \item $p\le 1-\delta$
   \item $q\le n^{-\gamma}$
   \item $pqn\le\epsilon\log n$.
  \end{itemize}
Let $G_1$ be a graph of order $n$ with maximal degree at most $qn$ and let $G_2\in\mathcal G(n,p)$.  Then with failure probability
$O(n^{-K})$ there is a packing of $G_1$ and $G_2$.
\end{theorem}

The rest of the paper is organized as follows.  In Section \ref{graphs} we prove Theorem \ref{extend}, 
and in Section \ref{hypergraphs} we prove the extension to hypergraphs.  
We conclude in Section \ref{conclusion} with some open problems.

\section{Packing random graphs}\label{graphs}

The aim of this section is to prove Theorem \ref{extend}.  We begin by noting a couple of standard facts.

We will use the following Chernoff-type inequalities.  Let $X$ be a sum of
Bernoulli random variables, and let $\mu=\E X$.  Then for $t>0$, we have
\begin{equation}\label{C1}
 \P[X\le\E X-t]\le\exp(-t^2/2\mu)
\end{equation}
and 
\begin{equation}\label{C2}
  \P[X\ge \E X+t]\le\exp(-t^2/(2\mu+2t/3))
\end{equation}
(see, e.g., \cite[Theorems 2.1 and 2.8]{JLR} or \cite[Chapter 2]{BLM}). 
Inequality \eqref{C2} is often called Bernstein's inequality.

It will also be useful to note a simple (and standard) fact about the
binomial distribution, see e.g., \cite[Corollary 2.4]{JLR}.

\begin{proposition}\label{bbound}
For every $K>0$ there is $\delta>0$ such that 
if $x>0$ and $X\sim \Bi(n,p)$ is a binomial random variable  with 
$np \le \delta x$ then $\P[X\ge x] \le e^{-Kx}$.
\end{proposition}

\begin{proof}
This is standard; we include a proof for completeness.
We have,
assuming as we may that $x$ is an integer, 
\begin{align*}
\P[X\ge x] 
&\le \binom {n}{x}p^{x}
\le \left(\frac{e n  p}{x}\right)^{x}
\le (e \delta)^{x}
\end{align*}
where we have used the standard bound $\binom nk\le(en/k)^k$ in the second line.
The result follows by choosing $\delta=e^{-K-1}$.
\end{proof}

Our first lemma is the following, which shows that, if $\mathcal A$ is a large, sparse set system then a random set (of suitable size)  is quite likely to be disjoint from some member of $\mathcal A$.

\begin{lemma}\label{disjoint}
For all $\delta,\gamma\in(0,1)$ there is $\epsilon>0$ such that the following holds for all sufficiently large $n$.  
Let $d=n^{1-\gamma}$, let $X$ be any set, and let $\mathcal A$ be a set sequence in $\mathcal P(X)$ such that:
\begin{itemize}
\item $|\mathcal A|\ge n$
\item every element of $X$ belongs to at most $d$ sets from $\mathcal A$ 
\item all sets in $\mathcal A$ have size at most $\epsilon \log n$.
\end{itemize}
Let $B\subset X$ be a random set where each element of $X$ independently belongs to $B$
with probability $1-\delta$.  Then 
$B$ is disjoint from at least $n^{1-\gamma/4}$ sets of $\mathcal A$, 
with failure probability $O(\exp(-n^{\gamma/3}))$.
\end{lemma}

\begin{proof}
This can be proved in more than one way (an alternative proof pointed out by a referee runs an element exposure martingale on X 
and then applies the Hoeffding-Azuma inequality).

We may assume that $|\mathcal A|=n$.  We choose a small $\epsilon>0$, and
assume that $n$ is large. 
We ignore below insignificant roundings to integers. 

We begin by partitioning $\mathcal A$ into sets of pairwise disjoint elements.  
Let $G$ be the intersection graph of $\mathcal A$: so the vertices of $G$ are the elements of $\mathcal A$, 
and $G$ has edges $AA'$ whenever $A\cap A'$ is nonempty.
Since every vertex belongs to at most $d$ sets from $\mathcal A$, and every set has size at most $\epsilon \log n$, each set in $\mathcal A$ meets at most $\epsilon d\log n$ other sets. 
Thus $G$ has maximal degree at most $\epsilon d\log n$.  It follows by a theorem of Hajnal and Szemer\'edi \cite{HS} that $G$ has a colouring with
at most $\epsilon d\log n +1$ colours in which the sizes of distinct colour classes differ by at most 1.   Thus we may partition $G$ into independent sets
(and so $\mathcal A$ into collections of pairwise disjoint sets) of
size at least $n/(\epsilon d\log n +1)\ge n^{\gamma/2}$.

Let $\mathcal A'$ be one of these collections of pairwise disjoint sets, and set $m=|\mathcal A'|\ge n^{\gamma/2}$.  
The random set $B$
is disjoint from each member of $\mathcal A'$ independently with probability at least $\delta^{\epsilon\log n}=n^{-\epsilon\log(1/\delta)}>n^{-0.01\gamma}$ 
provided we have chosen a sufficiently small $\epsilon$; 
it follows that the probability that $B$ is disjoint from fewer than
$m/n^{\gamma/4}$ sets in $\mathcal A'$ is at most 
\begin{align*}
\binom{m}{m/n^{\gamma/4}}(1-n^{-0.01\gamma})^{m-m/n^{\gamma/4}}
&\le\left(\frac{em}{m/n^{\gamma/4}}\right)^{m/n^{\gamma/4}} \exp(-n^{-0.01\gamma}m/2)\\
&<e^{m\log n/n^{\gamma/4}}e^{-n^{-0.01\gamma}m/2}\\
&<e^{-n^{-0.01\gamma}m/4},
\end{align*}
provided $n$ is sufficiently large.  There are $\epsilon d\log n +1=o(n)$ colour classes, so
with failure probability 
$o(ne^{-n^{-0.01\gamma} n^{\gamma/2}/4})=O(e^{-n^{\gamma/3}})$, 
$B$ is disjoint from at least a fraction $n^{-\gamma/4}$ of the sets in each colour class, and hence is disjoint from at least
$n^{1-\gamma/4}$ sets in $\mathcal A$.
\end{proof}

For positive integers $m,n$, and $p\in[0,1]$ we write $\mathcal S(n,m,p)$ for 
a random sequence $(S_i)_{i=1}^m$ of $m$ subsets of $[n]$,
where the subsets are independent and each set independently contains each element of $[n]$ with probability $p$.
Equivalently, we could consider a random $m \times n$ matrix with entries 0 and 1, where each element independently takes value 1 with probability $p$.   We shall refer to 
$S\in\mathcal S(n,m,p)$ as a {\em random set sequence}.

Given two random set sequences $\mathcal A\in\mathcal S(m,n,p)$ and $\mathcal A'\in\mathcal S(m,n,q)$, where $m\le n$,
it will be useful to pair up the sets from $\mathcal A$ and $\mathcal A'$ so that each pair is disjoint.  
For $A\in \mathcal A$ and $A'\in\mathcal A'$, the probability that $A$ and $A'$ are disjoint is $(1-pq)^n\le \exp(-npq)$, so if $pq>2\log n/n$ 
it is likely that we do not have any disjoint pairs at all.  However, if $pq<c\log n/n$, for small enough $c$, we will show that such a pairing is possible.
In fact we will prove a much stronger result: we can take just one of the set systems to be random, provided the other satisfies certain sparsity conditions.

\begin{lemma}\label{bipa}
For all $K>0$ and $\eta,\gamma,\delta\in(0,1)$ there is $\epsilon>0$ such that the following holds for all sufficiently large $n$.
Suppose that $p=p(n),q=q(n)\in [0,1]$ satisfy $0\le p<1-\delta$ and $pq<\epsilon \log n/n$.  Let $m\in[n^\eta,n]$ be an integer
and set $d=m^{1-\gamma}$, and  
suppose that $\mathcal A=(A_i)_{i=1}^m$ is a sequence of subsets of $[n]$ such that
\begin{itemize}
\item every $i\in[n]$ belongs to at most $d$ sets from $\mathcal A$
\item $\max_{A\in\mathcal A}|A|\le qn$.
\end{itemize}
Let $\mathcal B=(B_i)_{i=1}^m\in S(n,m,p)$
be a random set sequence, and
let $H$ be the bipartite graph with vertex classes $\mathcal A$ and $\mathcal B$, where we join $A_i$ to $B_j$ if 
$A_i\cap B_j=\emptyset$.  Then,
with failure probability $O(n^{-K})$, $H$ has minimal degree at least $m^{1-\gamma/4}$; furthermore,
$H$ has a perfect matching.
\end{lemma}

\begin{proof}
Let $\epsilon, \epsilon'>0$ be fixed, small quantities (with $\epsilon\ll \epsilon'$) that we shall choose later.
We generate
$\mathcal B$ in two steps:~we first choose a random set sequence
$\mathcal B'=(B_i')_{i=1}^m\in S(n,m,(1+\delta)p)$, and then obtain $\mathcal B$ from $\mathcal B'$ by deleting each
element from each set $B_i'$ independently with probability $\delta'=\delta/(1+\delta)$.

Note first that for any $i,j$, the distribution of the intersection $|A_i\cap B_j'|$ is stochastically dominated by a binomial
$\Bi(nq,p(1+\delta))$.  
So for fixed $\epsilon'>0$,  it follows from Proposition \ref{bbound} that we have 
$|A_i\cap B_j'|<\epsilon'\log m$ for all $i$ and $j$, with failure probability $O(n^{-K})$, provided $\epsilon$ is small enough
in terms of $\epsilon'$.
We may therefore assume from now on that this event occurs, and condition on the choice of $\mathcal B'$
(so $\mathcal B'$ is fixed and $\mathcal B$ is still random).

Now consider the bipartite graph $H$.
We need to prove that $H$ has a perfect matching.
We shall apply Hall's condition to $\mathcal B$, so it is enough to show that for every subset $S\subset \mathcal B$
we have $|\Gamma_H(S)|\ge|S|$.

Consider $B_i'\in\mathcal B'$, and let 
$\mathcal A'_i=(A_j\cap B'_i)_{j=1}^m$ be the restriction of $\mathcal A$ to $B'_i$.   Then 
every vertex belongs to at most $d$ sets from $\mathcal A'_i$
and $\max_j|A_j\cap B_i'|<\epsilon'\log m$, so provided $\epsilon'$ is sufficiently small
we can apply  Lemma \ref{disjoint} to deduce that
with failure probability $O(e^{-m^{\gamma/3}})$ the set $B_i$ is disjoint from
at least  $m^{1-\gamma/4}$ sets from $\mathcal A'_i$.  
This occurs independently for each $i$ (recall that we are conditioning on $\mathcal B'$), so
with failure probability
$O(m e^{-m^{\gamma/3}})=O(n^{-K})$ every vertex in $\mathcal B$ has degree at least $m^{1-\gamma/4}$ in $H$,
and so Hall's condition holds for every $S\subset \mathcal B$ with $|S|<m^{1-\gamma/4}$.

Now consider an element $A_i\in\mathcal A$.  Each $B_j'$ meets $A_i$ in at most $\epsilon'\log m$ vertices, and 
so each $B_j$ independently is disjoint from $A_i$ with probability at least 
$(\delta')^{\epsilon'\log m}>m^{-\gamma/6}$, provided $\epsilon'$ is sufficiently small.
The number of $B_j$ disjoint from $A_i$ is thus a binomial with expectation at least $m^{1-\gamma/6}$ and so,
by \eqref{C1}, is at least $m^{1-\gamma/6}/2>m^{1-\gamma/4}$,
with failure probability $O(e^{-m^{1-\gamma/6}/8})$.
So with failure probability 
$O(me^{-m^{1-\gamma/6}/8})=O(n^{-K})$
every vertex in $\mathcal A$ has degree at least $m^{1-\gamma/4}$ in $H$,
and so Hall's condition holds for every $S\subset \mathcal B$ with $|S|>m-m^{1-\gamma/4}$.

We have now shown that $H$ has minimal degree at least $m^{1-\gamma/4}$.  All that remains is
to verify Hall's condition for 
sets $S\subset\mathcal B$ of size between $m^{1-\gamma/4}$ and $m-m^{1-\gamma/4}$.
Let $t\in[m^{1-\gamma/4},m-m^{1-\gamma/4}]$: we shall bound the probability that there is any subset of $\mathcal B$ of size $t$ 
with $t$ or fewer neighbours in $\mathcal A$.  Suppose that $S\subset \mathcal B$ has size $t$ and
$T\subset \mathcal A$ has size $m-t$.  For any fixed $B_i'\in S$, 
the set sequence $\mathcal A'=(A\cap B_i')_{A\in T}$ has 
$\max_{A'\in\mathcal A'}|A'|\le\epsilon'\log m$ and
every vertex belongs to at most $d$ sets from $\mathcal A'$,
where $d=m^{1-\gamma}\le|\mathcal A'|^{1-\gamma/4}$.
So by Lemma \ref{disjoint}, the probability that
$B_i$ intersects every set in $T$ is at most $\exp(-(m-t)^{\gamma/12})$.  Thus the probability that (in the graph $H$) 
$S$ has no neighbours in $T$ is at most 
$\exp(-t\cdot(m-t)^{\gamma/12})$.  Since there are at most $n^{2t}=\exp(2t\log n)$ choices for the pair
$(S,T)$, we deduce that the probability that there is any set $S$ of size
$t$ with at most $t$ neighbours is bounded by
$\exp(2t\log n)\exp(-t\cdot(m-t)^{\gamma/12})=O(n^{-(K+1)})$, uniformly in
$t$.   
Summing over $t$, we see that Hall's
condition holds with failure probability $O(n^{-K})$.

We conclude by noting that we can choose first $\epsilon'$ and then $\epsilon$ sufficiently small for the estimates above to hold.
\end{proof}

We are now ready to prove  Theorem \ref{extend}.

\begin{proof}[Proof of Theorem \ref{extend}.]
Let $\eta=\gamma/2$, $t=\lceil (K+2)/\eta\rceil$, and let
$G_1$ have vertex set $V$ and $G_2$ have vertex set $W$.  We begin by finding a partition of $V$ into sets
$V_1,V_2,\dots$ of size $\Theta(n^\eta)$ such that:
\begin{itemize}
 \item $V_i$ is an independent set in $G_1$ for every $i$,
 \item Every vertex in $V$ has fewer than 
$t$ neighbours in each set $V_j$.
\end{itemize}
Indeed, we first colour $V$ randomly with $n^{1-\eta}$ colours, giving each vertex 
a colour selected uniformly at random and independently.  It follows from \eqref{C1} and \eqref{C2} that, with failure probability $O(n^{-K})$,
every colour class has size $(1+o(1))n^{\eta}$.  Consider a vertex $v \in V$, say 
with degree $d$.  Then by assumption $d\le qn\le n^{1-\gamma}$. 
So the probability that $v$ has a set of $t$ neighbours,
all with the same colour, is at most 
$$\binom{d}{t}(1/n^{1-\eta})^{t-1}\le d^t n^{\eta t-t+1} 
\le n^{1+t\eta-t\gamma}=n^{1-t\eta}=O(n^{-(K+1)}).$$
It follows that, with failure probability
$O(n^{-K})$, no vertex has $t$ neighbours of the same colour.  
Each colour class now induces a subgraph with maximum degree less than $t$, so we can
apply the Hajnal-Szemer\'edi Theorem \cite{HS} to each class, splitting it into $O(t)$ independent sets of (almost) the same size.
The vertex classes are now independent, have size $\Theta(n^\eta)$,
and no vertex has $t$ neighbours in any other class.

Reordering if necessary, we may assume that $|V_1|\ge|V_2|\ge\cdots$.
Now let $W=W_1\cup W_2\cup\cdots$ be an arbitrary partition of $W$ (chosen before revealing $G_2$) 
such that $|W_i|=|V_i|$ for every $i$.
We construct a bijection between $V$ and $W$ 
that defines a packing 
(i.e., does not map any edge of $G_1$ to an edge of $G_2$)
by 
constructing suitable 
bijections between $V_i$ and $W_i$ for $i=1,2,\dots$.

For $i=1$, we choose an arbitrary bijection between $V_1$ and $W_1$.  
(Recall that $V_1$ is independent.) 
For $i>1$, we set 
$S_i=\bigcup_{j<i}V_j$ and $T_i=\bigcup_{j<i}W_j$, and suppose that we have found a bijection $\phi_i:S_i\to T_i$.  
The neighbourhoods of vertices in $V_i$ and $W_i$ define
set sequences $\mathcal A=(\Gamma(v)\cap S_i)_{v\in V_i}$ in $S_i$ and  
$\mathcal B=(\Gamma(v)\cap T_i)_{v\in W_i}$ in $T_i$,
and the bijection $\phi_i$ allows us to identify $S_i$ and $T_i$.
We now check that these two set sequences satisfy the conditions of Lemma \ref{bipa}, which we will 
then apply to obtain a bijection between $V_i$ and $W_i$.
Let 
\begin{align*}
\widetilde n&=|S_i|=|T_i|=\Theta(in^\eta),\\ 
\widetilde m&=|V_i|=|W_i|=\Theta(n^\eta),
\end{align*}  
and note that $|\mathcal A|=|\mathcal B|=\widetilde m$
and $\widetilde m\in[{\widetilde n}^{\eta/2},\widetilde n]$. 
By construction of the partition $(V_j)_{j\ge1}$, no vertex belongs to $t$ sets from $\mathcal A$, 
as each vertex in $S_i$ has fewer than $t$ neighbours in $V_i$.
Let $\widetilde q=\max_{A\in \mathcal A}|A|/\widetilde n$.
Each set in $\mathcal A$ has size at most
$qn$ and so  $\widetilde  q \le qn/\widetilde n = O(qn^{1-\eta}/i)$.  
The set sequence $\mathcal B$ is random
with $\mathcal B\in S(\widetilde n,\widetilde m,p)$, 
and depends only on the edges between $W_i$ and $T_i$. 
Furthermore,
$$p\widetilde q\widetilde n
\le p \cdot (qn/\widetilde n)\cdot \widetilde n
=pqn\le \epsilon\log n=O(\epsilon\log \widetilde n).$$
We can therefore
apply Lemma \ref{bipa}, to deduce that if $\epsilon$ is sufficiently small then
with failure probability $O(n^{-(K+1)})$ there is a bijection 
between the two set sequences for which the corresponding pairs are disjoint; 
this corresponds to a bijection between $V_i$ and $W_i$ so that there are no common edges 
in the bipartite graphs between $(V_i,S_i)$ and $(W_i,T_i)$ where $S_i$ and $T_i$
are identified by $\phi_i$.  Extending $\phi_i$ with this bijection,
we obtain a bijection $\phi_{i+1}:S_{i+1}\to T_{i+1}$.

It follows that, with failure probability $O(n^{-K})$, we succeed at every step and construct the desired bijection.
\end{proof}

Finally in this section, we note that Theorem \ref{extend} can be used to pack several random graphs.

\begin{corollary}\label{multiples}
Let $\gamma, K>0$, let $\delta\in(0,1)$, and let $t$ be a positive integer.  Then there exists $\epsilon>0$ such that the following holds.  Let $p_0(n),\dots,p_t(n)$ satisfy
  \begin{itemize}
   \item $\max_i p_i\le 1-\delta$
   \item $p_0\le n^{-\gamma}$
   \item $\max_{i<j}p_ip_jn\le\epsilon\log n$.
  \end{itemize}
Let $G_0$ be a graph of order $n$ with maximal degree at most $p_0n$ and, for $i=2,\dots,t$, let $G_i\in\mathcal G(n,p_i)$.  Then with failure probability
$O(n^{-K})$ there is a packing of $G_0,\dots,G_t$.
\end{corollary}

\begin{proof}
We may assume that $p_1\le\dots\le p_t$.  
Thus, by the second and third conditions above, we have 
$\sum_{i=0}^{t-1}p_i =O(n^{-\min\{\gamma,1/3\}})$.
We first pack $G_0$ and $G_1$, then add in the remaining graphs one at a time, applying Theorem \ref{extend} at each stage.
Thus at the $i$th stage we have packed $G_0,\dots,G_i$ to obtain a graph $H_i$: it follows easily
from Proposition \ref{bbound} that 
with high probability the maximum degree condition of Theorem \ref{extend}
is satisfied by $H_i$ (with a slightly smaller $\gamma$).  Provided $\epsilon$
is sufficiently small, we get that
with failure probability $O(n^{-K})$ we can pack $H_i$ with $G_{i+1}$.
\end{proof}

\section{Packing hypergraphs}\label{hypergraphs}

In this section, we will prove Theorem \ref{upfront}. 

\begin{proof}[Proof of Theorem \ref{upfront}]
Note that the case $k=2$ follows immediately from Theorem \ref{extend}, so we can assume $k\ge 3$.
Let $\eta=1/5$, $t=15k$, and let
$\epsilon, \epsilon'>0$ be small constants and $K,K'$ large constants; 
we will choose $\epsilon,\epsilon'$ and $K,K'$ later. 
(In fact, we will first choose $\epsilon'$; $K'$ will be determined by $\epsilon'$;
we then choose $K$ and finally $\epsilon$.)
We may assume that $q\le p$, and so in particular 
$q=O(\sqrt{\log n/n^{k-1}})<n^{-1/2}$ (for large $n$).  
We may also assume that $q\ge \epsilon\log n /n^{k-1}$ , or increase to this value. 

Our argument will follow a similar strategy to Theorem \ref{extend}, but there are some additional complications.  It will be helpful to reveal the edges
of $G_1$ and $G_2$ in several steps.
This time we let $V$ be the vertex set of $G_2$ and $W$ the vertex set of
$G_1$.

We first generate a partition of $V$ into sets
$V_1,V_2,\dots$ by colouring $V$ randomly with $n^{1-\eta}$ colours, giving each vertex 
a colour selected uniformly at random and independently.  
It follows from \eqref{C1} and \eqref{C2} that, with failure probability $o(1)$,
every colour class $V_i$ has size $(1+o(1))n^{\eta}$,
so we may assume that this holds. 
Reordering if necessary, we may assume that $|V_1|\ge|V_2|\ge\cdots$.
Let $W=W_1\cup W_2\cup\cdots$ be a random partition of $W$ such that $|W_i|=|V_i|$ for every $i$.
For $i\ge1$, we set 
$S_i=\bigcup_{j<i}V_j$ and $T_i=\bigcup_{j<i}W_j$ (note that $S_1=T_1=\emptyset$; 
also $S_L=V$ and $T_L=W$, where $L=n^{1-\eta}+1$). 

As before, we will construct a bijection between $V$ and $W$ by 
constructing bijections between $V_i$ and $W_i$ for $i=1,2,\dots$.
However, we need to be a little more careful than in the graph case, as there are more ways for edges to intersect the classes $V_i$ and $W_i$.
For $j=1,\dots,k$, and any $i$, we say that an edge is {\em of type $j$ for $V_i$ or $W_i$}
if it has $j$ vertices in $V_i$ or $W_i$, and the remaining $k-j$ vertices in $S_i$ or $T_i$. 

We now reveal all type 1 edges in $G_2$.
For a $(k-1)$-set $A\subset S_i$, 
the probability that $V_i$ contains $t$ vertices $v$ such that $A\cup\{v\}$
is an edge of $G_2$  is at most
$$ \binom{2n^\eta}{t}q^t = O(n^{\eta t-t/2}) =o(n^{-k}).$$
It follows that, with high probability, for every integer $i$ and every
$(k-1)$-set $A\subset S_i$, $V_i$ contains fewer than $t$ vertices that can be added to $A$ to obtain an edge of $G_2$.
In other words, each $(k-1)$-set in $S_i$ is contained in fewer than $t$ type 1 edges for $V_i$.
 
For each vertex $v\in V_i$, we define the {\em type 1 neighbourhood of $v$} to be the $(k-1)$-uniform hypergraph on $S_i$
with edge set 
$$\{A\subset S_i: A\cup \{v\}\mbox{ is a type 1 edge for $V_i$}\};$$ 
similarly, for vertices in $W_i$, the type 1 neighbourhood is a $(k-1)$-uniform hypergraph on $T_i$.

At the first step of the partitioning process, we take a random bijection between $V_1$ and $W_1$.  The expected number of common edges 
is at most $pqn^{k\eta}=o(1)$, and so with high probability there are no common edges.

Now consider a later stage of the partitioning process: suppose we have constructed a bijection 
$\phi_i:S_i\to T_i$ and wish to extend this to a bijection $\phi_i:S_{i+1}\to T_{i+1}$.  
In constructing our bijection, we will only consider edges of type 1 and 2; we will consider edges of type 3 at the end of the argument.  

We first consider type 1 edges in $V_i$ and $W_i$.  For each $v\in V_i$, we consider the type 1 neighbourhood of $v$ as a subset of $S_i^{(k-1)}$
(rather than as a $k$-uniform hypergraph on $S_{i+1}$). 
The collection of type 1 neighbourhoods of vertices in $V_i$ then defines a
set sequence $\mathcal A$ of subsets of $S_i^{(k-1)}$; similarly, 
the collection of type 1 neighbourhoods of vertices in $W_i$ defines a
set sequence $\mathcal B$ of subsets of $T_i^{(k-1)}$; 
and the bijection $\phi_i$ allows us to identify $S_i^{(k-1)}$ and $T_i^{(k-1)}$.
As in the proof of Theorem \ref{extend}, we wish to apply Lemma \ref{bipa}, so we need to check that its conditions are satisfied.

Let 
\begin{align*}
\widetilde n&=|S_i^{(k-1)}|=|T_i^{(k-1)}|=\Theta(i^{k-1}n^{\eta(k-1)}),\\
\tilde m&=|V_i|=|W_i|=(1+o(1))n^\eta,
\end{align*}  
and note that 
$|\mathcal A|=|\mathcal B|=\tilde m$ and $\tilde m\in [{\tilde n}^{\eta/k},\tilde n]$.

By construction of the partition $(V_j)_{j\ge1}$, 
no element of $S_i^{(k-1)}$  is contained in $t$ sets from $\mathcal A$, 
as each $(k-1)$-set $A\subset S_i$ 
is contained in fewer than $t$ type 1 edges for $V_i$.   The size of each set in $\mathcal A$ has
distribution $\Bi(\widetilde n,q)$.  
Choose a small $\epsilon'>0$, let $K'=2/(\eta\epsilon')$, and then choose a large $K$.
Let $\widetilde q=\max\{Kq,\epsilon'(\log \widetilde n)/\widetilde n\}$.
It follows from Proposition \ref{bbound} that,
provided $K$ is large enough (depending on $K'$),
every set in $\mathcal A$ has size at most $\widetilde n\widetilde q$, 
with failure probability at most
\begin{equation*}
  \widetilde m e^{-K'\widetilde n\widetilde q}
\le n e^{-K'\epsilon'\log{\widetilde n}} 
\le n^{1-K'\epsilon'\eta} =o(1/n).
\end{equation*}
Furthermore, since $\widetilde n\le n^{k-1}$, by choosing $\epsilon$ small enough we get
\begin{equation*}
  p Kq \le K\epsilon \frac{\log n^{k-1}}{n^{k-1}}
\le \epsilon'  \frac{\log\widetilde n}{\widetilde n}
\end{equation*}
and hence $p\widetilde q\le \epsilon' (\log\widetilde n)/\widetilde n$.
We can therefore
apply Lemma \ref{bipa}, to deduce that if $\epsilon'$ is sufficiently small then
with failure probability $O(n^{-2})$ we get the following:
\begin{itemize}
\item a bijection $\phi^*:V_i\to W_i$ such that the corresponding pairs
in the two set sequences are disjoint.  This 
corresponds to a bijection between $V_i$ and $W_i$ so that there are no 
collisions between type 1 edges for $V_i$ and $W_i$.  Also: 
\item for all distinct $u,v\in V_i$ and $x,y\in W_i$, a bijection
$$\phi^{**}:V_i\setminus\{u,v\}\to W_i\setminus\{x,y\}$$
such that there are no collisions of type 1 edges for $V_i$ and $W_i$,
except possibly for edges containing $u$, $v$, $x$ or $y$.
\end{itemize}

The mapping $\phi^*$ deals with collisions between type 1 edges.
However, we must also consider type 2 edges for $V_i$ and $W_i$.
We do not reveal type 2 edges at this stage, but only the number of collisions between type 2 edges created by the mapping $\phi^*$.
There are at most $n^{k-2+2\eta}$ type 2 edges for $V_i$ and $W_i$, and so the probability that $\phi^*$ 
maps any type 2 edge for $V_i$ in $G_2$ to a type 2 edge in $G_1$
is at most $pqn^{k-2+2\eta}\le \log n /n^{1-2\eta}$; the probability that there are at least two collisions is
$O(\log^2n/n^{2-2\eta})=o(1/n)$ (which is small enough to ignore).   
If there are no collisions, then we use $\phi^*$ to extend $\phi_i$.  

This leaves the case when there is one collision between type 2 edges.  
We reveal the edge where this occurs: say $A\cup \{u,v\}$ maps to $A\cup
\{x,y\}$ under $\phi^*$. 
We thus condition on the existence of these two edges in $G_2$ and $G_1$,  
and on this being the only collision.  We shall show the existence of
another mapping $\phi^{**}$ from $V_i$ to $W_i$ that avoids collisions for
both type 1 and type 2 edges with probability at least $1-O(\log n/\sqrt n)$.
Then the probability that we get collisions for both $\phi^*$ and $\phi^{**}$
is $ O( (\log n /n^{1-2\eta})\cdot \log n/\sqrt n)$, which is $o(1/n)$.

Let $D=\lceil 6(\log \widetilde n)/\delta\rceil$.  We choose distinct vertices
$x_1,\dots,x_D,y_1\dots,y_D$ in $W_i$ such that the type 1 neighbourhood of
$u$ is edge-disjoint from the 
type 1 neighbourhoods of $x_1,\dots,x_D$, and the  type 1 neighbourhood of $v$ is edge-disjoint from the
type 1 neighbourhoods of $y_1,\dots,y_D$ 
(the existence of these vertices follows from the minimal degree condition on $H$ in Lemma \ref{bipa}).

We reveal the edges $A\cup\{x_\ell,y_\ell\}$ for each $\ell\le D$: 
since $p\le 1-\delta$, it follows that with probability $1-o(1/n)$ there is
some $\ell$ 
such that  $A\cup\{x_\ell,y_\ell\}$ is not present in $G_1$.  
We then use the appropriate mapping $\phi^{**}$ from 
$V_i\setminus\{u,v\}$ to $W_i\setminus\{x_\ell,y_\ell\}$ that we found
above, and 
extend it by setting $\phi^{**}(u)=x_\ell$ and $\phi^{**}(v)=y_\ell$
so that we have a mapping from $V_i$ to $W_i$.
The mapping $\phi^{**}$ does not cause any collision of type 1 edges.
Finally, we reexamine the type 2 edges for collisions.  
We have ensured that $A\cup\{u,v\}$  does not collide with anything; the probability of a collision involving any edge of form $A\cup\{x_j,y_j\}$
is at most $qD=O(\log n/\sqrt n)$; and the probability of any other
collision is at most $\log n /n^{1-2\eta}=O(1/\sqrt n)$, as before.
(More formally: we have conditioned on 
the edges $A\cup\{x_j,y_j\}$, on
the event that a particular pair of type 2 edges collide, and the event
that no other collisions occur.
If we resample all type 2 edges that are not in the colliding pair
or of form $A\cup\{x_j,y_j\}$, the number of collisions
under $\phi^{**}$ stochastically dominates the number before resampling, giving the same bound.)
So the probability that $\phi^{**}$ yields a collision is $O(\log n/\sqrt
n)$, as required.

It follows that, with probability $1-o(1/n)$, we are able to find a good  bijection between $V_i$ and $W_i$, and extend $\phi_i$ to $\phi_{i+1}$.  
Continuing in this way, we find a bijection from $V$ to $W$ in which there are no collisions between type 1 or 2 edges for any $V_i$, $W_i$.
 
Finally, we reveal all edges of type 3 or more.  There are at most $n^{k-2+2\eta}$ possible edges of type 3 or more, 
and so the probability that
any of these is an edge in both hypergraphs is at most $pqn^{k-2+2\eta}=o(1)$.  The algorithm therefore succeeds with probability $1-o(1)$.
 \end{proof}

\section{Conclusion}\label{conclusion}

We conclude by mentioning a few open questions.

\begin{itemize}
\item The bound in Theorem \ref{upfront} is sharp to within a constant factor.  It is natural to expect that there is some $c=c(k)>0$ such that
almost surely a pair of random $k$-uniform hypergraphs $G_1,G_2\in\mathcal G(n,k,p)$ 
are packable if $p<(c-\epsilon)\sqrt{\log n/n^{k-1}}$ and
are unpackable if $p>(c+\epsilon)\sqrt{\log n/n^{k-1}}$.
Is this correct?  If so, what is the value of $c$?
\item What happens with the results above if we take $G_1=G_2$?  We would expect this to make no difference.  
\item All our examples of unpackable $k$-uniform hypergraphs $G_1$, $G_2$
  have $\Delta(G_1)\Delta(G_2)=\Omega(n^{k-1})$.  
What is the correct bound here? 
\end{itemize}


\begin{thebibliography}{99}

\bibitem{A94}
N. Alon,
Packing of partial designs,
{\em Graphs and Combinatorics} {\bf 10} (1994), 11--18.

\bibitem{BE76} B. Bollob\'as and S. E. Eldridge, 
Maximal matchings in graphs with given maximal and minimal degrees, 
{\em Congressus Numererantium} {\bf XV} (1976), 165--168.

\bibitem{BS}
B. Bollob\'as and A. D. Scott, 
Intersections of random hypergraphs and tournaments,
{\em to appear}.

\bibitem{BLM}
S. Boucheron, 
G. Lugosi, 
and P. Massart, 
\emph{Concentration Inequalities},
Oxford Univ. Press, Oxford, 2013.

\bibitem{C74}
P.A. Catlin,
Subgraphs of graphs, I,
{\em Discrete Mathematics} {\bf 10} (1974), 225--233

\bibitem{C}
D. Conlon,
Hypergraph packing and sparse bipartite Ramsey numbers,
{\em Combinatorics Probability and Computing} {\bf 18} (2009), 913--923

\bibitem{HS} A. Hajnal and E. Szemer\'edi, 
Proof of a conjecture of P. Erd\H os, {\em in} Combinatorial theory and its
applications, II (Proc. Colloq., Balatonf\"{u}red, 1969),
601--623. North-Holland, Amsterdam, 1970. 

\bibitem{JLR}
S. Janson, T. \L uczak and A. Ruci\'nski,
\emph{Random Graphs},
Wiley, New York, 2000.

\bibitem{K14} 
P. Keevash, The existence of designs,
preprint, 2014, 
\texttt{arXiv:1401.3665}.

\bibitem{MNS}
J. Ma, H. Naves and B. Sudakov, 
Discrepancy of random graphs and hypergraphs, 
preprint, 2013, 
\texttt{arXiv:1302.3507}.

\bibitem{RRT}
V. R\"odl, A. Ruci\'nski and A. Taraz, Hypergraph Packing and Graph Embedding,
{\em Combinatorics, Probability and Computing} {\bf 8} (1999), 363--376.

\bibitem{SS78} 
N. Sauer and J. Spencer, Edge disjoint placement of graphs, {\em J. Combinatorial Theory Ser. B} {\bf 25} (1978), 295--302.

\bibitem{T77}
L. Teirlinck,
On making two Steiner triple systems disjoint, 
{\em J. Combinatorial Theory Ser. A} {\bf 23} (1977), 349--350. 

\end{thebibliography}
\end{document}